\documentclass[12pt, reqno]{amsart}
\usepackage{amsfonts}
\usepackage{bbm}
\usepackage{amscd,amsfonts}
\usepackage{amssymb, eucal, amsfonts, amsmath, xypic, latexsym}
\usepackage{pifont}
\usepackage{mathrsfs,color}
\usepackage{amsthm,indentfirst,bm,fancyhdr,dsfont}
\usepackage{graphicx}
\usepackage[all]{xy}
\usepackage[CJKbookmarks=true]{hyperref}

\usepackage{mathrsfs}
\usepackage{amsmath}
\usepackage{amssymb}
\usepackage{hyperref}

\newtheorem{theorem}{Theorem}[section]

\newtheorem{prop}[theorem]{Proposition}

\theoremstyle{definition}

\newtheorem{defn}[theorem]{Definition}

\newtheorem{remark}[theorem]{Remark}

\numberwithin{equation}{section}

\def\ggg{\mathfrak{g}}

\def\p{\mathfrak{p}}

\def\ggg{\mathfrak{g}}

\def\hhh{\mathfrak{h}}

\def\bbb{\mathfrak{b}}

\def\nnn{\mathfrak{n}}
\def\uuu{\mathfrak{u}}

\def\bbc{\mathbb{C}}
\def\bbf{\mathbb{F}}
\def\bbz{\mathbb{Z}}

\def\K{\mathbb{K}}

\def\0{\bar{0}}
\def\1{\bar{1}}

\def\GL{\text{GL}}
\def\Hom{\text{Hom}}

{\vskip-\lastskip\medskip
  \noindent
  {\em #1.}\enspace
  }%
{\qed\par\medskip
  }
\begin{document}

\title[The representations of the Lie superalgebra $\p(3)$ in prime characteristic]
{The representations of the Lie superalgebra $\p(3)$ in prime characteristic}
\author{Ye Ren}
\address{School of Mathematics, China University of Mining and Technology, Jiangsu 221116, China}\email{TBH560@cumt.edu.cn}


\begin{abstract} Let $\ggg$ be the Lie superalgebra $\p(3)$ of rank 2 over an algebraically closed field $\K$ of characteristic $p > 3$. We classify all irreducible modules of $\ggg$, and give the character formulae for irreducible modules.
\end{abstract}

\maketitle
\setcounter{tocdepth}{1}
\section*{Introduction}
The finite-dimensional simple Lie superalgebras over the field of complex numbers were classified by Kac in the 1970s (cf. \cite{K77}).  Although the classification of finite-dimensional simple Lie superalgebras over a field of prime characteristic has not been completed, the modular
version of the complex simple Lie superalgebras is still important. Recently, there has been an increasing interest in modular representation theory of restricted Lie superalgebras (cf. \cite{Du16, Sh23, WZ09, WZ11, Z24}, etc.). Nevertheless, their irreducible modules are not well-understood. The goal of this paper is to initiate the study of modular representations of the Lie superalgebra $\p(n)$ over an algebraically closed field $\K$ of characteristic $p$. The method for classifying the irreducible modules of $\p(3)$ is derived from \cite{Y13, Z09}.

This article is divided into two parts. The first part gives the general setup for the restricted Lie superalgebras $\p(n)$ and gives a proof that the Kac modules $K_{\chi}(\lambda)(\lambda\in \Lambda_{\chi})$ are irreducible for regular nilpotent $\chi$ when $p > n$. The second part gives a complete classification of simple $\p(3)$-modules and their character formulae for $p>3$.

\section{Preliminaries}
\subsection{}
Let $\K$ be an algebraically closed field of characteristic $p>2$. In this paper, it is the base field of the theory.
The concept of restricted Lie superalgebras given in \cite{WZ09} is a generalization of the concept of restricted Lie algebras (cf. \cite{J04}).
\begin{defn}
A Lie superalgebra $\ggg=\ggg_{\0}\oplus \ggg_{\1}$ is called a restricted Lie superalgebra, if there is a $p$th power map $\ggg_{\0}\rightarrow\ggg_{\0}$, denoted as $[p]$, satisfying:

(1) $(kx)^{[p]}= k^px^{[p]}$ for any $k\in \K$ and any $x\in \ggg_{\0}$,

(2) $[x^{[p]} ,y] = (\text{ad}x)^p(y)$ for any $x\in\ggg_{\0}$ and any $y\in\ggg$,

(3) $(x + y)^{[p]}= x^{[p]}+ y^{[p]}+ \sum\limits_{i=1}^{p-1}s_i(x,y)$ for any $x, y\in\ggg_{\0}$ where $is_i(x, y)$ is the coefficient of $\lambda^{i-1}$ in $(\text{ad}(\lambda x + y))^{p-1}(x)$.
\end{defn}
By definition, $x^p-x^{[p]}$ are in the center of the universal enveloping algebra $U(\ggg)$ for all $x\in\ggg_{\0}$. According to \cite[Proposition 2.3]{WZ09}, every irreducible $\ggg$-module $M$ is finite dimensional. So there is a $\chi\in \ggg_{\0}^*$, such that
$$(x^p-x^{[p]})_{|M}=\chi(x)^p\text{id}_{|M},$$
which is called a $p$-character of $M$. Let $I_{\chi}$ be the ideal of $U(\ggg)$ generated by $x^p-x^{[p]}-\chi(x)^p, x\in\ggg_0$. We call $U_{\chi}(\ggg)=U(\ggg)/I_{\chi}$ the reduced enveloping algebra of $\ggg$, call $U_{0}(\ggg)$ the restricted enveloping algebra of $\ggg$. The $U_{0}(\ggg)$-modules are usually called restricted modules.
\subsection{}
For $n>2$, the standard matrix realization of the periplectic Lie superalgebra is given as follows:
\[\widetilde{\p}(n)=\left\{X=\begin{pmatrix}
A & B \\
C & D \\
\end{pmatrix}|A=-D^t, B^t=B, C^t=-C\right\}
.\]
Here, $A, B, C, D$ are $n\times n$ matrices. The periplectic Lie superalgebra $\widetilde{\p}(n)$ is a restricted Lie superalgebra. Let the Lie superalgebra $\p(n)$ be the commutator of $\widetilde{\p}(n)$, so
\[\p(n)=\left\{X=\begin{pmatrix}
A & B \\
C & D \\
\end{pmatrix}|A=-D^t, B^t=B, C^t=-C, \text{tr}(A)=\text{tr}(B)=0\right\}
.\]
It is also a restricted Lie superalgebra. From now on, let $\ggg=\p(n)$. In this paper, we mainly consider the representation theory of $\p(n)$.

Let $E_{i,j}$ be the matrix of order $2n\times 2n$ which the $i$th row, $j$th column is $1$ and the rest are $0$, $1\leq i,j\leq n$. Let $\epsilon_i$ be the linear map $\epsilon_i(\sum\limits_{k=1}^na_kE_{k,k})=a_i, a_i\in \K$. Denote by
$$H_{\epsilon_i-\epsilon_j}= E_{i,i}-E_{j,j}-E_{n+i,n+i}+E_{n+j,n+j}, 1\leq i<j\leq n,$$
$$X_{\epsilon_i-\epsilon_j}= E_{i,j}-E_{n+j,n+i}, 1\leq i,j\leq n, i\neq j,$$
$$X_{-\epsilon_i-\epsilon_j}= E_{n+i,j}-E_{n+j,i}, 1\leq i< j\leq n,$$
$$X_{\epsilon_i+\epsilon_j}= E_{i,n+j}+E_{j,n+i}, 1\leq i\leq j\leq n.$$
Let $\hhh$ be the Cardan subalgebra of $\ggg$ spanned by $H_{\epsilon_i-\epsilon_{i+1}}(1\leq i\leq n-1)$. With respect to $\hhh$, there is a standard Cardan decomposition of $\ggg=\nnn^-\oplus\hhh\oplus\nnn$. Let $\ggg_{-1}$ be the subalgebra generated by $X_{-\epsilon_i-\epsilon_j}(1\leq i< j\leq n)$ and
$\ggg_{+1}$ be the subalgebra generated by $X_{\epsilon_i+\epsilon_j}(1\leq i\leq j\leq n)$. Let $\nnn^-=\nnn_0^-\oplus\ggg_{-1}$, $\nnn=\nnn_0\oplus\ggg_{+1}$.
So $\ggg_{\0}=\nnn_0^-\oplus\hhh\oplus\nnn_0$, and it is isomorphic to the special linear Lie algebra $\mathfrak{sl}(n)$. Denote by $\bbb_0=\nnn_0\oplus\hhh$, $\bbb=\nnn\oplus\hhh$.
Denote by $\Phi$ the root system. Let $\Phi^+ (\Phi^-)$ be the positive(negative) root system. Let $\Phi_0$ consist of even roots.
Then $\ggg=\hhh\oplus\bigoplus\limits_{\alpha\in\Phi}\ggg_{\alpha}$ is the root space decomposition of $\ggg$.
Denote by $\rho_0$ the half sum of the roots of $\nnn_0$. Denote by $W$ the Weyl group of $\Phi_0$. Let $s_{\alpha}\in W$ be the reflection correspond to $\alpha\in \Phi_0$. For $\lambda\in\hhh^*$, the dot action $s_{\alpha}\cdot\lambda=s_{\alpha}(\lambda+\rho_0)-\rho_0$.
The following proposition is a consequence of the PBW theorem for $U(\ggg)$.
\begin{prop}
Suppose $x_1,...,x_s$ is a basis of $\ggg_{\0}$, $y_1,...,y_t$ is a basis of $\ggg_{-1}$, $z_1,...,z_m$ is a basis of $\ggg_{+1}$. Then
$$\{y_1^{b_1}...y_t^{b_t}x_1^{a_1}...x_s^{a_s}z_1^{c_1}...z_m^{c_m}|0\leq a_i\leq p-1, b_j, c_k= 0,1 \text{ for all } i,j,k\}.$$
is a basis of $U_{\chi}(\ggg)$. In particular, dim$U_{\chi}(\ggg)= p^{\text{dim}\ggg_{\0}}2^{\text{dim}\ggg_{\1}}$.
\end{prop}
Let $U(\ggg)_d$ be spanned by above basis elements of $U(\ggg)$ with
$$d=\sum\limits_{k}c_k-\sum\limits_{j}b_j.$$ Then $$U(\ggg)=\bigoplus\limits_{d=k}^jU(\ggg)_d.$$ In this way, $U(\ggg)$ becomes a $\bbz$-grade algebra.
Denote by $d(x)$ the degree of the homogeneous element $x\in U(\ggg)$.

Since $\hhh$ is commutative, for any $\lambda\in\hhh^*$, there is a one-dimensional $U_{\chi}(\hhh)$-module $\K_{\lambda}$. For any $\chi$, define
$$\Lambda_{\chi}=\{\lambda\in\hhh^*| \lambda(h)^p-\lambda(h^{[p]})=\chi^p(h),\forall h\in\hhh\}.$$
Obviously, there are $p^{n-1}$ elements in $\Lambda_{\chi}$.

For $\chi\in \ggg_{\0}^*$, if $\gamma\in Aut_p(\ggg_{\0})$ (the group of automorphisms of $\ggg_{\0}$ as a restricted Lie algebra), then
$U_{\chi}(\ggg)\cong U_{\gamma\cdot\chi}(\ggg)$. Note that $\ggg_{\0}\cong\mathfrak{sl}(n)$, and there is a nature surjective homomorphism $\phi: \mathfrak{gl}(n)^*\rightarrow \ggg_{\0}^*$. So we only need to consider the $\GL(n)$-orbit of $\chi\in\ggg_0^*$.
There is a $\GL(n)$-invariant non-degenerate bilinear form on $\mathfrak{gl}(n)$  induced by trace, so $\mathfrak{gl}(n)\cong\mathfrak{gl}(n)^*$.  For each $\chi\in\ggg_{\0}^*$, there exists $g\in\GL(n)$ with $g\cdot\chi$ corresponds to a Jordan matrix. Then we can assume $\chi(\nnn_0^+)=0$. The $U_{\chi}(\hhh)$-module $\K_{\lambda}$ naturally becomes the $U_{\chi}(\bbb_0)$-module with trivial $\nnn_0$ action. Define the induced $U_{\chi}(\ggg_0)$-module：
$$Z_{\chi}^0(\lambda)=\text{Ind}_{\bbb_0}^{\ggg_0}(\K_{\lambda}).$$
By \cite[10.2]{J98}, $Z_{\chi}^0(\lambda)$ has a unique maximal submodule. Denote by $L_{\chi}^0(\lambda)$ the simple quotient of $Z_{\chi}^0(\lambda)$.
Let $\ggg_{+1}$ trivially acts on the simple $U_{\chi}(\ggg_{\0})$-module $L_{\chi}^0(\lambda)$. Then $L_{\chi}^0(\lambda)$ becomes a $U_{\chi}(\ggg_{\0}\oplus\ggg_{+1})$-module. Define Kac module：
$$K_{\chi}(\lambda)=\text{Ind}_{\ggg_{\0}\oplus\ggg_{+1}}^{\ggg}(L_{\chi}^0(\lambda)).$$

\begin{prop}
Let $\chi$ correspond to a Jordan matrix. Then $K_{\chi}(\lambda)$ has a unique maximal submodule, $\lambda\in\Lambda_{\chi}$. Denote by $L_{\chi}(\lambda)$ the simple quotient of $K_{\chi}(\lambda)$. Each simple $U_{\chi}(\ggg)$-module is isomorphic to one $L_{\chi}(\lambda)$ for some $\lambda\in\Lambda_{\chi}$.
\end{prop}

\begin{proof}
This proof is a modification of the proof given in \cite[3.9]{Du16}. As $U_{\chi}(\ggg_{\0}\oplus\ggg_{-1})$-module, $K_{\chi}(\lambda)\cong U_{\chi}(\ggg_{\0}\oplus\ggg_{-1})\otimes_{U_{\chi}(\ggg_{\0})}L_{\chi}^0(\lambda)$.
According to Frobenius reciprocity,
$$\Hom_{U_{\chi}(\ggg_{\0}\oplus\ggg_{-1})}(K_{\chi}(\lambda), L_{\chi}^0(\lambda))\cong \Hom_{U_{\chi}(\ggg_{\0})}(L_{\chi}^0(\lambda), L_{\chi}^0(\lambda)).$$
Since dim$\Hom_{U_{\chi}(\ggg_{\0})}(L_{\chi}^0(\lambda), L_{\chi}^0(\lambda))=1$, dim$\Hom_{U_{\chi}(\ggg_{\0}\oplus\ggg_{-1})}(K_{\chi}(\lambda), L_{\chi}^0(\lambda))$=1. So $\phi\in \Hom_{U_{\chi}(\ggg_{\0}\oplus\ggg_{-1})}(K_{\chi}(\lambda), L_{\chi}^0(\lambda))$ satisfies $\phi(xv)=ax\cdot v$ for some $a\in \K$. Here $x\in U_{\chi}(\ggg_{\0}\oplus\ggg_{-1}), v\in L_{\chi}^0(\lambda)$. So as $U_{\chi}(\ggg_{\0}\oplus\ggg_{-1})$-module, $K_{\chi}(\lambda)$ has a unique maximal submodule $N=\ggg_{-1}U_{\chi}(\ggg_{-1})L_{\chi}^0(\lambda)$.

Suppose $M$ and $M'$ are two different maximal $U_{\chi}(\ggg)$-submodules of $K_{\chi}(\lambda)$. Then $M$ and $M'$ are proper $U_{\chi}(\ggg_{\0}\oplus\ggg_{-1})$-submodules of $K_{\chi}(\lambda)$. So $K_{\chi}(\lambda)=M+M'\subset N$, a contradiction. Therefore, $K_{\chi}(\lambda)$ has a unique maximal $U_{\chi}(\ggg)$-submodule.
\end{proof}



For $\lambda=\sum\limits_{i=1}^n\lambda_i\epsilon_i\in\hhh^*$, set $$\delta(\lambda)=\prod\limits_{i<j}(\lambda_i-\lambda_j+j-i-1)\in\K.$$
Suppose $\delta(\lambda)\neq 0$. Then $\lambda$ is called a typical weight. Otherwise, it is called an atypical weight. Then we have the following proposition:
\begin{prop}(Refer to \cite[Theorem 5.5]{Z24})\label{delta}
Suppose $\delta(\lambda)\neq 0$. Then $K_{\chi}(\lambda)$ is a irreducible $U_{\chi}(\ggg)$-module.
\end{prop}

We call $\chi\in\ggg_{\0}^*$ is regular semisimple if there exists $g\in \GL(n)$ such that $g\cdot\chi(\nnn_0\oplus\nnn_0)=0$ and
$g\cdot\chi(H_{\epsilon_i-\epsilon_j})\neq 0$ for $1\leq i<j \leq n$.
We call $\chi\in\ggg_{\0}^*$ is regular nilpotent if there exists $g\in \GL(n)$ such that $g\cdot\chi(\hhh_0\oplus\nnn_0)=0$ and
$g\cdot\chi(X_{\alpha})\neq 0$ for all simple roots $\alpha$.

\begin{prop}\label{delta1}
(Refer to \cite[Porpositon 2.7]{RB23}) Suppose $\chi$ is regular semisimple. Then $K_{\chi}(\lambda)$ is irreducible.
\end{prop}

Note that the Weyl group $W$ is isomorphic to the $n$-symmetric group $S_n$. For $w\in W$, let $w=(i_1,i_2,...,i_n)$ if $w(\sum\limits_{k=1}^n x_i\epsilon_i)=\sum\limits_{i=1}^n x_{i_k}\epsilon_{i}$, $1\leq k,i_k\leq n$.
Then a direct computation shows
$$\delta(w\cdot \lambda)=\prod\limits_{k<s}(x_{i_k}-x_{i_s}+i_s-i_k-1).$$
\begin{prop}\label{delta2}
Let $\chi$ be regular nilpotent and $\lambda,\mu\in \Lambda_{\chi}$. Suppose there exists $w\in W$ such that $\delta(w\cdot\lambda)\neq 0$. Then $K_{\chi}(\lambda)$ is a irreducible $U_{\chi}(\ggg)$-module.
\end{prop}

\begin{proof}
By \cite[C.3]{J04}, $L^0_{\chi}(\lambda)\cong L^0_{\chi}(w\cdot\mu)$. So $K_{\chi}(\lambda)\cong K_{\chi}(w\cdot\lambda)$. By Proposition \ref{delta}, $K_{\chi}(w\cdot\lambda)$ is irreducible since $\delta(w\cdot\lambda)\neq 0$. Thus, $K_{\chi}(\lambda)$ is also irreducible.
\end{proof}

\begin{prop}\label{delta3}
Let $\chi$ be regular nilpotent and $\lambda\in\Lambda_{\chi}$. Suppose $p\geq n+1$. Then $K_{\chi}(\lambda)$ is a irreducible $U(\ggg)$-module.
\end{prop}

\begin{proof}
By Proposition \ref{delta2}, we only need to prove that there exists $w\in W$ such that $\delta(w\cdot\lambda)\neq 0$, or equivalently, there is no $\lambda\in\Lambda_{\chi}$ satisfies the system of equations $\delta(w\cdot\lambda)=0$ for all $w\in W$. Let $\lambda=\sum\limits_{i=1}^n x_i\epsilon_i\in\hhh^*$. Then $ x_i-x_j\in \bbf_{p}$.

We will prove it by induction on $n$. Firstly, it is easy to verify that the statement is true when $n=3$. There is a claim:
The system of equations $\delta(w\cdot\lambda)=0$ for all $w\in W$ is equivalent to the system of equations
$$\prod\limits_{1\leq i\leq n,i\neq k}(x_{i}-x_{k}+k-i-1)=0 \text{ for } 1\leq k\leq n.$$
The proof of the claim:
Let
\[
\bar{S_k}=\left\{(i_1,i_2,...,i_{k-1},k,i_{k+1},...,i_n)\mid
\begin{array}{cc}
(i_1,i_2,...,i_{k-1},i_{k+1},...,i_n) \text{ is a }\\
\text{ permutation of } (1,2,...,k-1,k+1,...,n)\\
\end{array}
\right\},
\]
$$L_k=\prod\limits_{1\leq i\leq n,i\neq k}(x_{i}-x_{k}+k-i-1).$$
Note that $L_{k}$ is the common monomial of $\delta(w\cdot\lambda), w\in \bar{S_k}$. By induction there is no $\lambda\in\Lambda_{\chi}$ satisfies the system of equations $\delta(w\cdot\lambda)/L_{k}=0$ for all $w\in \bar{S_k}$. So $L_k$ must be zero.

Clearly, the system of equations $L_k=0$  for $1\leq k\leq n-1$ is equivalent to
\[\begin{cases}
(x_{i_1}-x_n+n-i_1-1)=0\\
(x_{i_2}-x_{i_1}+i_1-i_2-1)(x_{i_2}-x_{n}+n-i_2-1)=0\\
(x_{i_3}-x_{i_2}+i_2-i_3-1)(x_{i_3}-x_{i_1}+i_1-i_3-1)(x_{i_3}-x_n+n-i_3-1)=0\\
\quad\vdots\\
(x_{i_{n-1}}-x_{i_{n-2}}+i_{n-2}-i_{n-1}-1)\cdots(x_{i_{n-1}}-x_{n}+n-i_{n-1}-1)=0
\end{cases}
\]
for $i_1,...,i_{n-1}$ being all the possible permutations of $1,...,n-1$.
It shows that $x_i-x_n-i\in\{-1,...,-(n-1)\}, 1\leq i\leq n-1$.
It contradicts the fact that $L_n=0$ for $p\geq n+1$.
Thus, there must exists $w\in W$ such that $\delta(w\cdot\lambda)\neq 0$ for any $\lambda\in\Lambda_{\chi}$. So in such conditions, $K_{\chi}(\lambda)$ is irreducible.
\end{proof}
\begin{remark}
Note that $p\geq n+1$ is a sufficient condition in the above proposition. And there is another proof given in \cite{RB23}.
\end{remark}

\section{The representations of $\p(3)$}
In the following part of the article, let $\ggg=\p(3)$ and $p>3$. We will investigate the character formula of the irreducible $\ggg$-module $L_{\chi}(\lambda)$, $\lambda\in\Lambda_{\chi}$. Write $\lambda=(r, s)$ if $\lambda(H_{\epsilon_1-\epsilon_2})=r, \lambda(H_{\epsilon_2-\epsilon_3})=s$. There are $p^2$  possibilities for $\lambda$. We only need to consider the $\GL(3)$-orbit of $\chi\in\ggg_0^*$. The representatives of orbits are given below (refer to \cite{XW17}):
\begin{itemize}
\item[(1)] $\chi_1$ is regular semisimple, i.e. $\chi_1(H_{\epsilon_1-\epsilon_2} )\chi_1(H_{\epsilon_2-\epsilon_3})\chi_1(H_{\epsilon_1-\epsilon_3})\neq 0$, $\chi_1(\nnn_0\oplus\nnn_0^-)=0$.

\item[(2)] $\chi_2$ is subregular semisimple, i.e. $\chi_2(H_{\epsilon_1-\epsilon_2})=0$, $\chi_2(H_{\epsilon_2-\epsilon_3})\chi_2(H_{\epsilon_1-\epsilon_3})\neq 0$, $\chi_2(\nnn_0\oplus\nnn_0^-)=0$.

\item[(3)] $\chi_3=0$.

\item[(4)] $\chi_4(H_{\epsilon_2-\epsilon_3})\neq 0$, $\chi_4(X_{-\epsilon_1+\epsilon_2})=1$, $\chi_4(\nnn_0)=\chi_4(H_{\epsilon_1-\epsilon_2})=\chi_4(X_{-\epsilon_1+\epsilon_3})=\chi_4(X_{-\epsilon_2+\epsilon_3})=0$.

\item[(5)] $\chi_5$ is subregular nilpotent, i.e. $\chi_5(H_{\epsilon_1-\epsilon_2})=\chi_5(H_{\epsilon_2-\epsilon_3})=0$, $\chi_5(X_{-\epsilon_2+\epsilon_3})=1$, $\chi_5(\nnn_0)=0=\chi_5(X_{-\epsilon_1+\epsilon_2})=\chi_5(X_{-\epsilon_1+\epsilon_3})=0$.

\item[(6)] $\chi_6$ is regular nilpotent, i.e. $\chi_6(X_{-\epsilon_1+\epsilon_2} )=\chi_6(X_{-\epsilon_2+\epsilon_3})=1$, $\chi_6(\hhh\oplus\nnn_0)=0$.
\end{itemize}

For convenience, we fix the notations of the orbits given above.
Let $\mathfrak{u}^-=\ggg_{-\epsilon_2+\epsilon_3}\oplus\ggg_{-\epsilon_1+\epsilon_3}$ and $\mathfrak{l}=\hhh\oplus\ggg_{\epsilon_1-\epsilon_2}\oplus\ggg_{\epsilon_2-\epsilon_1}\cong \mathfrak{gl}(2)$.
The corresponding irreducible $U_{\chi}(\ggg_0)$-modules are given below:
\begin{itemize}
\item[(1)]  Let $\chi=\chi_1$. Then $L^0_{\chi}(\lambda)\cong Z_{\chi}^0(\lambda)$.

\item[(2)]  Let $\chi=\chi_2$. Then $L^0_{\chi}(\lambda)\cong U_{\chi}(\mathfrak{u}^-)\otimes V_{\chi}$, where
$V_{\chi}$ is the irreducible $U_0(\mathfrak{l})$-module with highest weight $\lambda$. The irreducible $\mathfrak{l}\cong\mathfrak{gl}(2)$-modules is showed in \cite{J98}.

\item[(3)] Let $\chi=\chi_3$. Denote by $L(\lambda)$ the irreducible $\mathfrak{sl}(3,\bbc)$-module with the maximal weight $\lambda$. Denote by $\bar{L}(\lambda) (\lambda\in \Lambda_{\chi})$ the irreducible $\mathfrak{sl}(3,\K)$-module corresponding to $L(\lambda)$ constructed in \cite{B67}. Then \cite[Theorem 4]{B67} shows:
\begin{itemize}
\item[\tiny(3.1)] Suppose $r+s\leq p-2$. Then
$L_{\chi}^0(\lambda)\cong \bar{L}(\lambda)$.

\item[\tiny(3.2)] Suppose $\lambda=(p-1, 0)$ or $\lambda=(0, p-1)$. Then
$L_{\chi}^0(\lambda)\cong \bar{L}(\lambda)$.

\item[\tiny(3.3)] Suppose $r+s=p-1 (s,r\geq 1) $.
Then $L_{\chi}^0(\lambda)\cong \bar{L}(\lambda)/\bar{L}(\lambda-2\epsilon_1-\epsilon_2)$.
\end{itemize}
\item[(4)] Let $\chi=\chi_4$. Then $L^0_{\chi}(\lambda)\cong U_{\chi}(\mathfrak{u}^-)\otimes V_{\chi}$, where
$V_{\chi}$ is the irreducible $U_{\nu}(\mathfrak{l})$-module with highest weight $\lambda$, $\nu\in\mathfrak{gl}(2)^*$ is regular nilpotent.
Note that $L^0_{\chi}(\lambda)\cong L^0_{\chi}(s_{\epsilon_1-\epsilon_2}\cdot\lambda).$

\item[(5)] Let $\chi=\chi_5$. Then $L^0_{\chi}(\lambda)\cong U_{\chi}(\mathfrak{u}^-)\otimes V_{\chi}$, where
$V_{\chi}$ is the irreducible $U_{0}(\mathfrak{l})$-module with highest weight $\lambda$ (cf. \cite{J99}).

\item[(6)] Let $\chi=\chi_6$. Then $L^0_{\chi}(\lambda)\cong Z_{\chi}^0(\lambda)$.
\end{itemize}
\begin{defn}
A vector $v\in K_{\chi}(\lambda)$ is called a maximal vector of weight $\mu\in\hhh^*$ if $h\cdot v=\mu(h)v$ for any $h\in\hhh$, and $\nnn\cdot v=0$.
\end{defn}

The character formula of $K_{\chi}(\lambda)$ is clear. So to investigate the character formula of $L_{\chi}(\lambda)$, it is sufficient to investigate the composition factors of $K_{\chi}(\lambda)$. A way to investigate the composition factors of $K_{\chi}(\lambda)$ is to find all the maximal vectors of $K_{\chi}(\lambda)$.

Recall that $U(\ggg)$ is a $\bbz$-graded algebra, and so is $U_{\chi}(\ggg)$.
Let $v=u_1v_1+u_2v_2+u_3v_3$, $v_1, v_2, v_3\in L_{\chi}^0(\lambda),$ $u_1, u_2, u_3\in U(\ggg_{-1})$ with $d(u_1)=-1, d(u_2)=-2$ and $d(u_3)=-3.$
Since $K_{\chi}(\lambda)$ is $U_{\chi}(\ggg_{-1})$-free. Let $$K_{\chi}(\lambda)_{-i}=\text{Span}_{\K}\{uv|d(u)=-i, v\in L_{\chi}^0(\lambda)\}.$$
Then $K_{\chi}(\lambda)$ becomes a $\bbz$-graded $U_{\chi}(\ggg)$-module. Denote by $d(x)$ the degree of a homogeneous element $x\in K_{\chi}(\lambda)$.

\begin{prop}\label{max v}
A vector $v\in K_{\chi}(\lambda)$ is a maximal vector of weight $\mu\in\hhh^*$ if and only if $u_1v_1,u_2v_2,u_3v_3$ are maximal vectors of weight $\mu$.
\end{prop}
\begin{proof}
Let $x\in\ggg_{+1}$. Then $d(xu_1)=0$, $d(xu_2)=-1$, $d(xu_3)=-2$. So $xv=0$ if and only if $xu_iv_i=0, i=1,2,3$. Let $y\in\nnn_0$. Then $d(yu_1)=-1$, $d(yu_2)=-2$, $d(yu_3)=-3$. So $yv=0$ if and only if $yu_iv_i=0, i=1,2,3$. If $h\in\hhh$, then
$hv-\mu(h)v=(h-\mu(h))u_1v_1+(h-\mu(h))u_2v_2+(h-\mu(u))u_3v_3$. Since $K_{\chi}(\lambda)$ is $U_{\chi}(\ggg_{-1})$-free, $u_1v_1,u_2v_2,u_3v_3$ are linearly independent in $K_{\chi}(\lambda)$. So $hv-\mu(h)v=0$ if and only if $hu_iv_i=\mu(h)u_iv_i, i=1,2,3$.
\end{proof}
By Proposition \ref{max v}, we only need to find the homogeneous maximal vectors.
The following three propositions are obtained by a direct computation.
\begin{prop}\label{m1}
Let
$m_1=X_{-\epsilon_1-\epsilon_2}w_1+X_{-\epsilon_1-\epsilon_3}w_2+X_{-\epsilon_2-\epsilon_3}w_3\in K_{\chi}(\lambda)$, $w_1$, $w_2$, $w_3\in L_{\chi}^0(\lambda).$ Then
$m_1$ is a maximal vector of weight $(\mu_1,\mu_2)$ if and only if
\begin{gather}
H_{\epsilon_1-\epsilon_2}w_1=\mu_1w_1, H_{\epsilon_1-\epsilon_2}w_2=(\mu_1+1)w_2, H_{\epsilon_1-\epsilon_2}w_3=(\mu_1-1)w_3;\label{2.1}\\
H_{\epsilon_2-\epsilon_3}w_1=(\mu_2+1)w_1, H_{\epsilon_2-\epsilon_3}w_2=(\mu_2-1)w_2, H_{\epsilon_1-\epsilon_2}w_3=\mu_2w_3;\label{2.2}\\
X_{-\epsilon_1+\epsilon_3}w_2+X_{-\epsilon_2+\epsilon_3}w_3=0;\label{2.3}\\
X_{\epsilon_1-\epsilon_2}w_1=X_{\epsilon_1-\epsilon_2}w_2=-w_2+X_{\epsilon_1-\epsilon_2}w_3=0;\label{2.4}\\
X_{\epsilon_2-\epsilon_3}w_1=-w_1+X_{\epsilon_2-\epsilon_3}w_2=X_{\epsilon_2-\epsilon_3}w_3=0.\label{2.5}
\end{gather}
\end{prop}

\begin{prop}\label{m2}
Let $m_2=X_{-\epsilon_1-\epsilon_3}X_{-\epsilon_1-\epsilon_2}w_1+X_{-\epsilon_2-\epsilon_3}X_{-\epsilon_1-\epsilon_2}w_2 +X_{-\epsilon_1-\epsilon_3}X_{-\epsilon_2-\epsilon_3}$ $w_3\in K_{\chi}(\lambda)$, $w_1, w_2, w_3\in L_{\chi}^0(\lambda).$ Then
$m_2$ is a maximal vector of weight $(\mu_1,\mu_2)$ if and only if
\begin{gather}
H_{\epsilon_1-\epsilon_2}w_1=(\mu_1+1)w_1, H_{\epsilon_1-\epsilon_2}w_2=(\mu_1-1)w_2, H_{\epsilon_1-\epsilon_2}w_3=\mu_1w_3;\label{2.6}\\
H_{\epsilon_2-\epsilon_3}w_1=\mu_2w_1, H_{\epsilon_2-\epsilon_3}w_2=(\mu_2+1)w_2, H_{\epsilon_1-\epsilon_2}w_3=(\mu_2-1)w_3;\label{2.7}\\
X_{-\epsilon_2+\epsilon_3}w_2+X_{-\epsilon_1+\epsilon_3}w_1+w_3=X_{-\epsilon_1+\epsilon_3}w_3=
X_{-\epsilon_2+\epsilon_3}w_3=0;\label{2.8}\\
X_{\epsilon_1-\epsilon_2}w_1=X_{\epsilon_1-\epsilon_2}w_2-w_1=X_{\epsilon_1-\epsilon_2}w_3=0;\label{2.9}\\
X_{\epsilon_2-\epsilon_3}w_1=X_{\epsilon_2-\epsilon_3}w_2=w_2+X_{\epsilon_2-\epsilon_3}w_3=0.\label{2.10}
\end{gather}
\end{prop}
\begin{prop}\label{m3}
Let $m_3=X_{-\epsilon_1-\epsilon_3}X_{-\epsilon_2-\epsilon_3}X_{-\epsilon_1-\epsilon_2}w\in V$, $w_1, w_2, w_3\in L_{\chi}^0(\lambda).$ Then
$m_3$ is a maximal vector of weight $(\mu_1,\mu_2)$ if and only if
\begin{gather}
H_{\epsilon_1-\epsilon_2}w=\mu_1w, H_{\epsilon_2-\epsilon_3}w=\mu_2w_2;\label{2.11}\\
X_{-\epsilon_2+\epsilon_3}w=X_{-\epsilon_1+\epsilon_3}w=X_{\epsilon_1-\epsilon_2}w=X_{\epsilon_2-\epsilon_3}w=0.\label{2.12}
\end{gather}
\end{prop}
We will investigate irreducible modules for different $\chi$ case by case.  Firstly, by Proposition \ref{delta1} and Proposition \ref{delta3}, $K_{\chi}(\lambda)$ is irreducible for regular nilpotent and regular semisimple $\chi$.
Then we investigate irreducible modules for $\chi=\chi_2$ subregular semisimple. Suppose $\lambda=(r,s)\in\Lambda_{\chi}$. Then $r\in\bbf_p, s\notin\bbf_p$.
As $\delta(\lambda)=rs(r+s+1)$, by Proposition \ref{delta}, we only need to consider the case when $r=0$.

\begin{prop}
Let $\chi=\chi_2$ be subregular semisimple and $\lambda\in\Lambda_{\chi}$. Suppose  $\lambda=(0, s), s\notin\bbf_p$. Then $$[K_{\chi}(\lambda)]=[L_{\chi}(\lambda)]+[L_{\chi}((\lambda-\epsilon_1-\epsilon_2)].$$
\end{prop}
\begin{proof}
Let $m_1, m_2, m_3$ given as in Proposition \ref{m1}, Proposition \ref{m2} and Proposition \ref{m3} be $\bbb$-maximal vectors. Let $v$ be the unique nonzero maximal vector of $L_{\chi}^0(\lambda)$ up to a scalar.
By (\ref{2.4}) and (\ref{2.5}), $X_{\epsilon_1-\epsilon_2}w_1=X_{\epsilon_2-\epsilon_3}w_1=0$. So $w_1$ is $\nnn_0$-maximal. Suppose $w_1\neq 0$. We may assume $w_1=v$. The weight of $w_1$ is $(0, s)$. By (\ref{2.1}) and (\ref{2.2}), the weight of $w_2$ is $(1,s-2)$. Let $$w_2=\sum\limits_{a,b,c} k_{c,b,a}X_{-\epsilon_1+\epsilon_3}^{c}X_{-\epsilon_2+\epsilon_3}^{b}X_{-\epsilon_1+\epsilon_2}^{a}v, k_{c,b,a}\in\K.$$
Comparing the weights of $w_1$ and $w_2$, we have
\[ \left\{
     \begin{array}{ll}
       -c+b-2a=1,  \\
       -c-2b+a=-2.
     \end{array}
\right.\]
So
$$w_2=\sum\limits_{0\leq a\leq p-1} k_{-a,a+1,a}X_{-\epsilon_1+\epsilon_3}^{-a}X_{-\epsilon_2+\epsilon_3}^{a+1}X_{-\epsilon_1+\epsilon_2}^{a}v.$$
(Remark:  We also use $-a$ to represent its smallest nonnegative integer modulo $p$).
As $\lambda=(0, s)$, $L_{\chi}^0(\lambda)$ is $U_{\chi}(\uuu^-)$-free (refer to \cite{J98}). So $a$ must be zero, and $w_2=k_{0,1,0}X_{-\epsilon_2+\epsilon_3}v$. By (\ref{2.5}), $X_{\epsilon_2-\epsilon_3}w_2= k_{0,1,0}sv=v$. So $$w_2=\frac{1}{s}X_{-\epsilon_2+\epsilon_3}v.$$
By (\ref{2.1}) and (\ref{2.2}), the weight of $w_3$ is $(-1, s-1)$. So
$$w_3=\sum\limits_{0\leq a\leq p-1} t_{-a+1,a,a}X_{-\epsilon_1+\epsilon_3}^{-a+1}X_{-\epsilon_2+\epsilon_3}^{a}X_{-\epsilon_1+\epsilon_2}^{a}v=t_{1,0,0}X_{\epsilon_3-\epsilon_1}v, t_{-a+1,a,a}\in\K$$
By (\ref{2.4}), $X_{\epsilon_1-\epsilon_2}w_3=-t_{1,0,0}X_{\epsilon_3-\epsilon_2}v=w_2=\frac{1}{s}X_{-\epsilon_2+\epsilon_3}v$. So
$$w_3=-\frac{1}{s}X_{\epsilon_3-\epsilon_1}v.$$
Then it can be checked that $$m_1=X_{-\epsilon_1-\epsilon_2}v+\frac{1}{s}X_{-\epsilon_1-\epsilon_3}X_{-\epsilon_2+\epsilon_3}v-\frac{1}{s}X_{-\epsilon_2-\epsilon_3}X_{-\epsilon_1+\epsilon_3}v$$
satisfies (\ref{2.1})-(\ref{2.5}).

Suppose $w_1=0$. By (\ref{2.4}) and (\ref{2.5}), $w_2$ is a nonzero multiple of $v$ if $w_2\neq 0$. We may assume $w_2=v$. By (\ref{2.1}) and (\ref{2.2}), the weight of $w_3$ is $(-2,s+1)$. Then $w_3=t_{1,p-1,0}X_{-\epsilon_1+\epsilon_3}X_{-\epsilon_2+\epsilon_3}^{p-1}v$. As $X_{-\epsilon_1+\epsilon_3}w_2+X_{-\epsilon_2+\epsilon_3}w_3=X_{-\epsilon_1+\epsilon_3}v\neq 0$, it contradicts (\ref{2.3}). So $w_2=0$.
By (\ref{2.4}) and (\ref{2.5}), $w_3$ is a nonzero multiple of $v$ if $w_3\neq 0$. We may assume $w_3=v$. It contradicts (\ref{2.3}). So $w_1=w_2=w_3=0$ and $m_1=0$.
Applying a similar argument, by (\ref{2.6})-(\ref{2.12}), it can be checked that $m_2=m_3=0$.

So $v$ and $X_{-\epsilon_1-\epsilon_2}v+\frac{1}{s}X_{-\epsilon_1-\epsilon_3}X_{-\epsilon_2+\epsilon_3}v-\frac{1}{s}X_{-\epsilon_2-\epsilon_3}X_{-\epsilon_1+\epsilon_3}v$ are the only two maximal vectors of $K_{\chi}(\lambda)$ up to a scalar.
\end{proof}

\begin{prop}\label{chi5}
Let $\chi=\chi_4$ and $\lambda\in\Lambda_{\chi}$. Suppose  $\lambda$ is atypical. Then
$$[K_{\chi}(\lambda)]=[L_{\chi}(\lambda)].$$
\end{prop}
\begin{proof}
Suppose $m_1, m_2, m_3$ given as in Proposition \ref{m1}, Proposition \ref{m2} and Proposition \ref{m3} are $\bbb$-maximal vectors. Let $v$ be a nonzero maximal vector of $L_{\chi}^0(\lambda)$.
By (\ref{2.4}) and (\ref{2.5}), $w_1$ is $\nnn_0$-maximal. Suppose $w_1\neq 0$. We may assume $w_1=v$.
Note that
$$\{X_{-\epsilon_1+\epsilon_3}^{c}X_{-\epsilon_2+\epsilon_3}^{b}X_{-\epsilon_1+\epsilon_2}^{a}v|0\leq a,b,c\leq p-1\}$$
is a basis of $L_{\chi}^0(\lambda)$.
Suppose the weight of $w_1$ is $\mu=(\mu_1,\mu_2)$. Actually, $\mu$ equals $\lambda$ or $s_{\epsilon_1-\epsilon_2}\cdot\lambda$. By (\ref{2.1}) and (\ref{2.2}), the weight of $w_2$ is $(\mu_1+1,\mu_2-2)$.
So $$w_2=\sum\limits_{0\leq a\leq p-1} k_{-a,a+1,a}X_{-\epsilon_1+\epsilon_3}^{-a}X_{-\epsilon_2+\epsilon_3}^{a+1}X_{-\epsilon_1+\epsilon_2}^{a}v.$$
As
\begin{eqnarray*}
&&X_{\epsilon_2-\epsilon_3}w_2=X_{\epsilon_2-\epsilon_3}\sum\limits_{0\leq a\leq p-1}(k_{-a,a+1,a})X_{-\epsilon_1+\epsilon_3}^{-a}X_{-\epsilon_2+\epsilon_3}^{a+1}X_{-\epsilon_1+\epsilon_2}^{a}v\\
&=&\sum\limits_{0\leq a\leq p-1}(k_{-a+1,a,a-1}(-a+1)+k_{-a,a+1,a}(a+1)(a+\mu_2))X_{-\epsilon_1+\epsilon_3}^{-a}X_{-\epsilon_2+\epsilon_3}^{a}X_{-\epsilon_1+\epsilon_2}^{a}v\\
&=&v
\end{eqnarray*}
So
\[\left(
  \begin{array}{c}
    k_{1,0,p-1}+\mu_2k_{0,1,0} \\
    0+2(\mu_2+1)k_{p-1,2,1} \\
    -k_{p-1,2,1}+3(\mu_2+2)k_{p-2,3,2} \\
    \vdots\\
    3k_{3,p-2,p-3}-(\mu_2-2)k_{2,p-1,-2} \\
    2k_{2,p-1,p-2}+0
  \end{array}
\right)
=
\left(
  \begin{array}{c}
    1 \\
    0 \\
    0 \\
    \vdots \\
    0 \\
    0 \\
  \end{array}
\right)
.\]
So $w_2=k_{0,1,0}X_{-\epsilon_2+\epsilon_3}v+k_{1,0,p-1}X_{-\epsilon_1+\epsilon_3}X_{-\epsilon_1+\epsilon_2}^{p-1}v$, and it satisfies
$k_{1,0,p-1}+k_{0,1,0}\mu_2=1.$  By (\ref{2.4}), $X_{\epsilon_1-\epsilon_2}w_2=0
=-k_{1,0,p-1}(\mu_1+2)X_{-\epsilon_1+\epsilon_3}X_{-\epsilon_1+\epsilon_2}^{p-2}v-k_{1,0,p-1}X_{-\epsilon_2+\epsilon_3}X_{-\epsilon_1+\epsilon_2}^{p-1}v$.
So $k_{1,0,p-1}=0$, and
$$w_2=\frac{1}{\mu_2}X_{-\epsilon_2+\epsilon_3}v.$$
By (\ref{2.1}) and (\ref{2.2}), the weight of $w_3$ is $(\mu_1-1,\mu_2-1)$. So
$$w_3=\sum\limits_{c=-a+1,b=a} t_{c,b,a}X_{-\epsilon_1+\epsilon_3}^{c}X_{-\epsilon_2+\epsilon_3}^{b}X_{-\epsilon_1+\epsilon_2}^{a}v.$$
By (\ref{2.4}),
\begin{eqnarray*}
&&X_{\epsilon_1-\epsilon_2}w_3=X_{\epsilon_1-\epsilon_2}\sum\limits_{0\leq a\leq p-1} t_{-a+1,a,a}X_{-\epsilon_1+\epsilon_3}^{-a+1}X_{-\epsilon_2+\epsilon_3}^{a}X_{-\epsilon_1+\epsilon_2}^{a}v\\
&=&\sum\limits_{0\leq a\leq p-1}(t_{-a,a+1,a+1}(a+1)(\mu_1-a)+t_{-a+1,a,a}(a-1))
X_{-\epsilon_1+\epsilon_3}^{-a}X_{-\epsilon_2+\epsilon_3}^{a+1}X_{\epsilon_2-\epsilon_1}^{a}v\\
&=&\frac{1}{\mu_2}X_{-\epsilon_2+\epsilon_3}v
\end{eqnarray*}
So
\[\left(
  \begin{array}{c}
   t_{0,1,1}(\mu_1)+t_{1,0,0}(-1) \\
    t_{-1,2,2}(2)(\mu_1-1)+t_{0,1,1}(0) \\
    t_{-2,3,3}(3)(\mu_1-2)+t_{-1,2,2}(1) \\
    \vdots\\
    t_{2,-1,-1}(-1)(\mu_1+2)+t_{3,-2,-2}(-3) \\
   t_{1,0,0}(0)(\mu_1+1)+t_{2,-1,-1}(-2)
  \end{array}
\right)
=
\left(
  \begin{array}{c}
    \frac{1}{\mu_2} \\
    0 \\
    0 \\
    \vdots \\
    0 \\
    0 \\
  \end{array}
\right)
\]
So $w_3= t_{0,1,1}X_{-\epsilon_2+\epsilon_3}X_{-\epsilon_1+\epsilon_2}v+
t_{1,0,0}X_{-\epsilon_1+\epsilon_3}v$ and satisfies  $t_{0,1,1}(\mu_1)+t_{1,0,0}(-1)=\frac{1}{\mu_2}$.
By (\ref{2.3}),
\begin{eqnarray*}
&&X_{-\epsilon_1+\epsilon_3}w_2+X_{-\epsilon_2+\epsilon_3}w_3\\
&=&\frac{1}{\mu_2}X_{-\epsilon_1+\epsilon_3}X_{-\epsilon_2+\epsilon_3}v+t_{0,1,1}X_{-\epsilon_2+\epsilon_3}^2X_{-\epsilon_1+\epsilon_2}v+
t_{1,0,0}X_{-\epsilon_1+\epsilon_3}X_{-\epsilon_2+\epsilon_3}v\\
&=&0
\end{eqnarray*}
So $$w_3=-\frac{1}{\mu_2}X_{-\epsilon_1+\epsilon_3}v.$$
As $X_{\epsilon_2-\epsilon_3}w_3
=-\frac{1}{\mu_2}X_{-\epsilon_1+\epsilon_2}v\neq 0$, it contradicts (\ref{2.5}). So $w_1=0$.

Applying a similar argument, it can be checked that $w_1=w_2=w_3=0$. So $m_1=0$.
By (\ref{2.6})-(\ref{2.12}), we can also check that $m_2=m_3=0$. So $K_{\chi}(\lambda)$ is irreducible.

\end{proof}

\begin{prop}
Let $\chi=\chi_5$ be subregular nilpotent, and $\lambda\in\Lambda_{\chi}$. Suppose  $\lambda$ is atypical. Then

\begin{itemize}
\item[(1)] If $\lambda=(0,s),s\neq 0,p-1$.
$$[K_{\chi}(\lambda)]=[L_{\chi}(\lambda)]+[L_{\chi}(\lambda-\epsilon_1-\epsilon_2)].$$

\item[(2)] If $\lambda=(0,p-1)$.
$$[K_{\chi}(\lambda)]=[L_{\chi}(\lambda)]+[L_{\chi}(\lambda-\epsilon_1-\epsilon_2)]+[L_{\chi}(\lambda-\epsilon_1-\epsilon_3)].$$

\item[(3)] Otherwise,
$$[K_{\chi}(\lambda)]=[L_{\chi}(\lambda)].$$

\end{itemize}
\end{prop}
\begin{proof}
The proof is similar to the proof of Proposition \ref{chi5},
and we only list the maximal vectors of  $K_{\chi}(\lambda)$ up to a scalar:

 (1) If $\lambda=(0,s),s\neq 0,p-1$.
$$m_1=X_{-\epsilon_1-\epsilon_2}v+\frac{1}{s}X_{-\epsilon_1-\epsilon_3}X_{-\epsilon_2+\epsilon_3}v-\frac{1}{s}X_{-\epsilon_2-\epsilon_3}X_{-\epsilon_1+\epsilon_3}v,$$
    $$m_2=m_3=0.$$

 (2) If $\lambda=(0,p-1)$.
$$m_1=X_{-\epsilon_1-\epsilon_2}v+\frac{1}{s}X_{-\epsilon_1-\epsilon_3}X_{-\epsilon_2+\epsilon_3}v-\frac{1}{s}X_{-\epsilon_2-\epsilon_3}X_{-\epsilon_1+\epsilon_3}v,$$
$$m_1=X_{-\epsilon_1-\epsilon_3}v-X_{-\epsilon_2-\epsilon_3}X_{-\epsilon_1+\epsilon_3}X_{-\epsilon_2+\epsilon_3}^{p-1}v.$$
    $$m_2=m_3=0.$$
\end{proof}

\begin{prop}\label{chi0}
Let $\chi=0$ and $\lambda\in\Lambda_{\chi}$. Suppose $\lambda=(r, s)$ is atypical.
\begin{itemize}
\item[(1)]  If $r+s=p-1, s\neq 0,r\neq 0$, then
$$[K_{\chi}(\lambda)]=[L_{\chi}(\lambda)].$$
\item[(2)]  If $\lambda=(0,s)$, then
\[\left\{
    \begin{array}{ll}
      [K_{\chi}(\lambda)]=[L_{\chi}(\lambda)]+[L_{\chi}\lambda-\epsilon_1-\epsilon_2)], & \hbox{$s\geq 2$;} \\
      {[K_{\chi}(\lambda)]=[L_{\chi}(\lambda)]+[L_{\chi}(\lambda-\epsilon_1-\epsilon_2)]+[L_{\chi}(\lambda-\epsilon_2)]}, & \hbox{$s=1$;} \\
      {[K_{\chi}(0)]=[L_{\chi}(0)]+[L_{\chi}(-\epsilon_2-\epsilon_3)]+[L_{\chi}(-2\epsilon_1-2\epsilon_2-2\epsilon_3)]}, & \hbox{$s=0$.}
    \end{array}
  \right.
\]
\item[(3)]  If $\lambda=(r,0)$, then
 \[
\left\{
  \begin{array}{ll}
   [K_{\chi}(\lambda)]=[L_{\chi}(\lambda)]+[L_{\chi}\lambda-\epsilon_2-\epsilon_3)], & \hbox{$r\geq 2$;} \\
  {[K_{\chi}(\lambda)]=[L_{\chi}(\lambda)]+[L_{\chi}(\lambda-\epsilon_2-\epsilon_3)]+[L_{\chi}(\lambda-\epsilon_1)]}, & \hbox{$r=1$.}
  \end{array}
\right.
\]
\end{itemize}
\end{prop}
\begin{proof}
(1) Recall that $L_{\chi}^0(\lambda)\cong \bar{L}^0(\lambda)/\bar{L}^0(\lambda-2\epsilon_1-\epsilon_2)$, where $\bar{L}^0(\lambda)$ is the irreducible $\mathfrak{sl}(3,\K)$-module corresponding to the irreducible $\mathfrak{sl}(3,\bbc)$-module $L^0(\lambda)$ constructed in \cite{B67}. Let $v$ be a nonzero maximal vector of $\bar{L}^0(\lambda)$ of weight $\lambda$. Then $$X_{-\epsilon_2+\epsilon_3}X_{-\epsilon_1+\epsilon_2}v+rX_{-\epsilon_1+\epsilon_3}v$$ generates a maximal submodule of $\bar{L}^0(\lambda)$. Note that
$$\{X_{-\epsilon_1+\epsilon_3}^{c}X_{-\epsilon_2+\epsilon_3}^{b}X_{-\epsilon_1+\epsilon_2}^{a}v|0\leq a\leq r, 0\leq b\leq s, 0\leq c\leq r+s-a-b\}$$
form a basis of $\bar{L}^0(\lambda)$.
For any nonzero vector $w\in L_{\chi}^0(\lambda)$, write
$$w=\sum\limits_{0\leq a+b+c<p}k_{c,b,a}X_{-\epsilon_1+\epsilon_3}^{c}X_{-\epsilon_2+\epsilon_3}^{b}X_{-\epsilon_1+\epsilon_2}^{a}v,k_{c,b,a}\in\K.$$
Suppose $m_1, m_2, m_3$ given as in Proposition \ref{m1}, Proposition \ref{m2} and Proposition \ref{m3} are $\bbb$-maximal vectors.
By (\ref{2.4}) and (\ref{2.5}), $w_1$ is $\nnn_0$-maximal. Suppose $w_1\neq 0$. We may assume $w_1=v$.
The weight of $w_1$ is $\lambda=(r,s)$. By (\ref{2.1}) and (\ref{2.2}), the weight of $w_2$ is $(r+1,s-2)$ and the weight of $w_3$ is $(r-1,s-1)$.
So we may assume 
$$w_2=k_{0,1,0}X_{-\epsilon_2+\epsilon_3}v,$$ 
$$w_3=t_{1,0,0}X_{-\epsilon_1+\epsilon_3}v+t_{0,1,1}X_{-\epsilon_2+\epsilon_3}X_{-\epsilon_1+\epsilon_2}v.$$
By (\ref{2.5}), $X_{\epsilon_2-\epsilon_3}w_2=w_1$. So
$$w_2=\frac{1}{s}X_{-\epsilon_2+\epsilon_3}v.$$
By (\ref{2.3}), \begin{eqnarray*}
&&X_{-\epsilon_1+\epsilon_3}w_2+X_{-\epsilon_2+\epsilon_3}w_3\\
&=&\frac{1}{s}X_{-\epsilon_1+\epsilon_3}X_{-\epsilon_2+\epsilon_3}v+
t_{0,1,1}X_{-\epsilon_2+\epsilon_3}^2X_{-\epsilon_1+\epsilon_2}v+
t_{1,0,0}X_{-\epsilon_1+\epsilon_3}X_{-\epsilon_2+\epsilon_3}v\\
&=&0.
\end{eqnarray*}
So $t_{1,0,0}=-\frac{1}{s}$. By (\ref{2.4}),
$X_{\epsilon_1-\epsilon_2}w_3=(t_{0,1,1}r+\frac{1}{s})X_{-\epsilon_2+\epsilon_3}v=w_2$. So $t_{0,1,1}=0$, and
$$w_3=-\frac{1}{s}X_{-\epsilon_1+\epsilon_3}v.$$
As $X_{\epsilon_2-\epsilon_3}w_3=
-\frac{1}{s}X_{\epsilon_2-\epsilon_1}v\neq 0$, it contradicts (\ref{2.5}). So $w_1=0$.

Applying a similar argument, it can be checked that $w_1=w_2=w_3=0$. So $m_1=0$.
By (\ref{2.6})-(\ref{2.12}), we can also check that $m_2=m_3=0$. So $K_{\chi}(\lambda)$ is irreducible.

(2) For this case, the basis of $L_{\chi}^0(\lambda)$ is $$\{X_{-\epsilon_1+\epsilon_3}^{c}X_{-\epsilon_2+\epsilon_3}^{b}v|0 \leq b\leq s, 0 \leq c\leq s-b\}.$$
Then we just need to check:

(2.1) For $s\geq 2$, up to a scalar, $$m_1=X_{-\epsilon_1-\epsilon_2}v+\frac{1}{s}X_{-\epsilon_1-\epsilon_3}X_{-\epsilon_2+\epsilon_3}v-\frac{1}{s}X_{-\epsilon_2-\epsilon_3}X_{-\epsilon_1+\epsilon_3}v.$$
$$m_2=m_3=0.$$

(2.2) For $s=1$, up to a scalar, $$m_1=X_{-\epsilon_1-\epsilon_2}v+\frac{1}{s}X_{-\epsilon_1-\epsilon_3}X_{-\epsilon_2+\epsilon_3}v-\frac{1}{s}X_{-\epsilon_2-\epsilon_3}X_{-\epsilon_1+\epsilon_3}v.$$
$$m_2=X_{-\epsilon_2-\epsilon_3}X_{-\epsilon_1-\epsilon_2}v-X_{-\epsilon_1-\epsilon_3}X_{-\epsilon_2-\epsilon_3}X_{-\epsilon_2+\epsilon_3}v.$$
$$m_3=0.$$

(2.3) For $s=0$, up to a scalar,
$$m_1=X_{-\epsilon_2-\epsilon_3}v.$$
$$m_2=0.$$
$$m_3=X_{-\epsilon_1-\epsilon_3}X_{-\epsilon_2-\epsilon_3}X_{-\epsilon_1-\epsilon_2}.$$

(3) It can be proved in a similar way as (2).
\end{proof}
\begin{remark}
The characters of $K_{\chi}(\lambda)$ over $\K$ given in Proposition \ref{chi0} (2)(3) are the same as the characters of $K_{\chi}(\lambda)$ over $\bbc$ given in \cite{C15}.
\end{remark}
In summary of the arguments given above, we obtain the main theorem of this article:
\begin{theorem}
Keep notations as above, the multiplicities of simple modules of $K_{\chi}(\lambda)$ are given as follows:
\begin{itemize}
\item[(1)]  Suppose $\lambda\in\Lambda_{\chi}$ is typical. Then
$$[K_{\chi}(\lambda)]=[L_{\chi}(\lambda)].$$

\item[(2)]  Suppose $\chi=\chi_1,\chi_4$ or $\chi_6$. Then
$$[K_{\chi}(\lambda)]=[L_{\chi}(\lambda)].$$

\item[(3)]  Suppose $\chi=\chi_2$ and $\lambda=(0, s), s\neq 0$. Then
$$[K_{\chi}(\lambda)]=[L_{\chi}(\lambda)]+[L_{\chi}(\lambda-\epsilon_1-\epsilon_2)].$$

\item[(4)]  Suppose $\chi=\chi_3$ and $\lambda=(r, s)\in\Lambda_{\chi}$ be atypical.
\begin{itemize}
\item[(4.1)]  If $r+s=p-1, s\neq 0,r\neq 0$, then
$$[K_{\chi}(\lambda)]=[L_{\chi}(\lambda)].$$
\item[(4.2)]  If $\lambda=(0,s)$, then
\[\left\{
    \begin{array}{ll}
      [K_{\chi}(\lambda)]=[L_{\chi}(\lambda)]+[L_{\chi}\lambda-\epsilon_1-\epsilon_2)], & \hbox{$s\geq 2$;} \\
      {[K_{\chi}(\lambda)]=[L_{\chi}(\lambda)]+[L_{\chi}(\lambda-\epsilon_1-\epsilon_2)]+[L_{\chi}(\lambda-\epsilon_2)]}, & \hbox{$s=1$;} \\
      {[K_{\chi}(0)]=[L_{\chi}(0)]+[L_{\chi}(-\epsilon_2-\epsilon_3)]+[L_{\chi}(-2\epsilon_1-2\epsilon_2-2\epsilon_3)]}, & \hbox{$s=0$.}
    \end{array}
  \right.
\]
\item[(4.3)]  If $\lambda=(r,0)$, then
 \[
\left\{
  \begin{array}{ll}
   [K_{\chi}(\lambda)]=[L_{\chi}(\lambda)]+[L_{\chi}\lambda-\epsilon_2-\epsilon_3)], & \hbox{$r\geq 2$;} \\
  {[K_{\chi}(\lambda)]=[L_{\chi}(\lambda)]+[L_{\chi}(\lambda-\epsilon_2-\epsilon_3)]+[L_{\chi}(\lambda-\epsilon_1)]}, & \hbox{$r=1$.}
  \end{array}
\right.
\]
\end{itemize}

\item[(5)]  Suppose $\chi_5$ and $\lambda\in\Lambda_{\chi}$. Suppose $\lambda=(r, s)$ is atypical. Then
\begin{itemize}

\item[(5.1)] If $\lambda=(0,s),s\neq 0,p-1$.
$$[K_{\chi}(\lambda)]=[L_{\chi}(\lambda)]+[L_{\chi}(\lambda-\epsilon_1-\epsilon_2)].$$

\item[(5.2)] If $\lambda=(0,p-1)$.
$$[K_{\chi}(\lambda)]=[L_{\chi}(\lambda)]+[L_{\chi}(\lambda-\epsilon_1-\epsilon_2)]+[L_{\chi}(\lambda-\epsilon_1-\epsilon_3)].$$

\item[(5.3)] Otherwise,
$$[K_{\chi}(\lambda)]=[L_{\chi}(\lambda)].$$
\end{itemize}

\end{itemize}

\end{theorem}

\end{document}